\documentclass[12pt]{amsart}
\usepackage[left=25truemm,right=25truemm,top=25truemm,bottom=25truemm]{geometry}
\usepackage{latexsym,amssymb,amsmath,amsthm,amscd,graphicx,esint,pgf,xcolor}

\makeatletter
\@namedef{subjclassname@2010}{%
  \textup{2010} Mathematics Subject Classification}
\makeatother

\newtheorem{thm}{Theorem}[section]
\newtheorem{cor}[thm]{Corollary}
\newtheorem{lem}[thm]{Lemma}

\theoremstyle{definition}

\newtheorem{rem}[thm]{Remark}

\numberwithin{equation}{section}

\newcommand{\cA}{{\mathcal A}}
\newcommand{\cB}{{\mathcal B}}

\newcommand{\cD}{{\mathcal D}}

\newcommand{\cL}{{\mathcal L}}
\newcommand{\cM}{{\mathcal M}}

\newcommand{\cS}{{\mathcal S}}
\newcommand{\cT}{{\mathcal T}}

\newcommand{\R}{{\mathbb R}}

\newcommand{\Z}{{\mathbb Z}}

\def\al{\alpha}

\def\dl{\delta}
\def\Dl{\Delta}

\def\lm{\lambda}
\def\Lm{\Lambda}

\def\0{\emptyset}
\def\1{{\bf 1}}
\def\6{\partial}
\def\8{\infty}

\def\lt{\left}
\def\rt{\right}
\def\ds{\displaystyle}
\def\ol{\overline}

\begin{document}

\title{Choquet integrals, Hausdorff content and sparse operators}

\author[N.~Hatano]{Naoya Hatano}
\address{
Department of Mathematics, Chuo University, 1-13-27, Kasuga, Bunkyo-ku, Tokyo 112-8551, Japan.
}
\email{n.hatano.chuo@gmail.com}

\author[R.~Kawasumi]{Ryota Kawasumi}
\address{
Minohara 1-6-3 (B-2), Misawa, Aomori 033-0033, Japan
}
\email{rykawasumi@gmail.com}

\author[H.~Saito]{Hiroki Saito}
\address{
College of Science and Technology, Nihon University,
Narashinodai 7-24-1, Funabashi City, Chiba, 274-8501, Japan
}
\email{saitou.hiroki@nihon-u.ac.jp}

\author[H.~Tanaka]{Hitoshi~Tanaka}
\address{
Research and Support Center on Higher Education for the hearing and Visually Impaired, 
National University Corporation Tsukuba University of Technology,
Kasuga 4-12-7, Tsukuba City, Ibaraki, 305-8521 Japan
}
\email{htanaka@k.tsukuba-tech.ac.jp}

\thanks{
The first named author is financially supported by 
a~Foundation of Research Fellows, The Mathematical Society of Japan.
The third named author is supported by 
Grant-in-Aid for Young Scientists (19K14577), 
the Japan Society for the Promotion of Science. 
The forth named author is supported by 
Grant-in-Aid for Scientific Research (C) (15K04918 and 19K03510), 
the Japan Society for the Promotion of Science.
Part of this work was supported by the Research Institute for Mathematical Sciences,
an International Joint Usage/Research Center located in Kyoto University 
and had been inspired during the third and forth named authors visited NCTS Workshop on Harmonic Analysis 2019.
}



\date{}

\begin{abstract}
Let $H^d$, $0<d<n$, be the dyadic Hausdorff content of the $n$-dimensional Euclidean space $\R^n$. 
It is shown that $H^d$ counts a~Cantor set of the unit cube $[0, 1)^n$ as $\approx 1$, 
which implies unboundedness of the sparse operator $\cA_{\cS}$ on the Choquet space $\cL^p(H^d)$, $p>0$. 
In this paper 
we verify that the sparse operator $\cA_{\cS}$ maps 
$\cL^p(H^d)$, $1\le p<\8$, into 
an associate space of Orlicz-Morrey space
${\cM^{p'}_{\Phi_0}(H^d)}'$, 
$\Phi_0(t)=t\log(e+t)$. 
We also give another characterizations of 
those associate spaces using the tiling $\cT$ of $\R^n$.
\end{abstract}

\maketitle

\section{Introduction}\label{sec1}
The purpose of this paper is to study the boundedness properties of the sparse operator $\cA_{\cS}$ 
defined on the Choquet space $\cL^p(H^d)$, $p\ge 1$. 

\subsection*{The Choquet space $\cL^p(H^d)$}
Let $\cD$ be the class of dyadic cubes 
$Q=2^{-k}(j+[0,1)^n)$, 
$k\in\Z, j\in\Z^n$. 
For the set $E\subset\R^n$, 
the dyadic Hausdorff content $H^d$ of $E$ is defined by
$$
H^d(E)
=
\inf\lt\{
\sum_j\ell(Q_j)^d:\,
E\subset\bigcup_jQ_j,
Q_j\in\cD
\rt\},
$$
where, 
the infimum is taken over all coverings of $E$ by 
countable families of dyadic cubes $Q_j\in\cD$ 
and $\ell(Q_j)$ stands for the side length of $Q_j$.
The Choquet integral with respect to $H^d$ is defined by
$$
\int_{\R^n}f\,{\rm d}H^d
=
\int_0^{\8}H^d(\{f>t\})\,{\rm d}t,
\quad f\ge 0.
$$
For $0<p<\8$,
we define the Choquet space $\cL^p(H^d)$ by 
the completion of the set of all continuous function having compact support 
$C_0(\R^n)$ with the functional
$$
\|f\|_{\cL^p(H^d)}
=
\lt(\int_{\R^n}|f|^p\,{\rm d}H^d\rt)^{\frac1p}.
$$
We notice that for $p\ge 1$ 
the functional satisfies the triangle inequality 
and thus it becomes a~norm 
(cf. \cite{STW}).
We also define the Choquet space $\cL^{\8}(H^d)$ 
by the functional 
$$
\|f\|_{\cL^{\8}(H^d)}
=
\inf\{t>0:\,H^d(\{|f|>t\})=0\}.
$$

\subsection*{The Morrey space $\cM^p(H^d)$}
The maximal operator $M_d$ of a~(signed) Radon measure $\mu$ is defined by 
$$
M_d\mu(x)
=
\sup_{Q\in\cD}
\frac{|\mu|(Q)}{\ell(Q)^d}\1_{Q}(x),
\quad x\in\R^n,
$$
where, 
$\1_{E}$ denotes the characteristic function of the set $E$.
For $1<p\le\8$, 
we define the Morrey space $\cM^p(H^d)$ by 
the set of all Radon measures $\mu$ satisfying 
$$
\|\mu\|_{\cM^p(H^d)}
=
\|M_d\mu\|_{\cL^p(H^d)}<\8.
$$

\subsection*{Adams' dual theorem}
The following is known as Adams' dual theorem
(cf. \cite{Ad,ST}).

\begin{quotation}
For $1\le p<\8$, 
the dual space of $\cL^p(H^d)$ is $\cM^{p'}(H^d)$.
More precisely, 
if $F$ is a~bounded linear functional on $\cL^p(H^d)$, 
then there exists a~Radon measure 
$\mu\in\cM^{p'}(H^d)$ such that
$$
F(\phi)
=
\int_{\R^n}\phi\,{\rm d}\mu,
\quad \phi\in C_0(\R^n),
$$
and,
$$
\|F\|
\approx
\|\mu\|_{\cM^{p'}(H^d)}.
$$
Here, $p'=\frac{p}{p-1}$, 
for $1<p<\8$, and 
$p'=\8$, for $p=1$, 
denotes the conjugate exponent of $p$. 
\end{quotation}

This theorem can be proved by the following Adams' dual inequality. 

\begin{equation}\label{1.1}
\int_{\R^n}f\,{\rm d}|\mu|
\le
\int_{\R^n}fM_d\mu\,{\rm d}H^d,
\quad f\ge 0.
\end{equation}

\begin{proof}[Proof of \eqref{1.1}]
Set $G=fM_d\mu$. Then,
$$
\int_{\R^n}f\,{\rm d}|\mu|
=
\int_0^{\8}
\frac{|\mu|}{M_d\mu}(\{G>t\})
\,{\rm d}t.
$$
Letting $\{Q_j\}\subset\cD$ be an~any covering of $\{G>t\}$, 
we have that
$$
\frac{|\mu|}{M_d\mu}(\{G>t\})
\le
\sum_j
\frac{|\mu|}{M_d\mu}(Q_j)
\le
\sum_j\ell(Q_j)^d,
$$
where we have used a~simple trick
$$
\frac{|\mu|(Q_j)}{\ell(Q_j)^d}
\le
M_d\mu(x),\quad x\in Q_j.
$$
These entail 
$$
\int_{\R^n}f\,{\rm d}|\mu|
\le
\int_{\R^n}G\,{\rm d}H^d
=
\int_{\R^n}fM_d\mu\,{\rm d}H^d.
$$
\end{proof}

\subsection*{Sparse operators}
Following \cite{Le}, we introduce sparse operators.
By a~cube in $\R^n$ we mean a~half-open cube 
$Q=\prod_{i=1}^n[a_i, a_i+h)$, $h>0$. 
We say that a~family $\cS$ of cubes from $\R^n$ is $\eta$ sparse, $0<\eta<1$, 
if for every $Q\in\cS$, 
there exists a~measurable set 
$E_{Q}\subset Q$ such that 
$|E_{Q}|\ge\eta|Q|$, and 
the sets $\{E_{Q}\}_{Q\in\cS}$ 
are pairwise disjoint. 
Given a~sparse family $\cS$, 
define a~sparse operator $\cA_{\cS}$ by
$$
\cA_{\cS}f(x)
=
\sum_{Q\in\cS}
\fint_{Q}f\cdot\1_{Q}(x),
\quad x\in\R^n,
$$
where, the barred integral $\fint_{Q}f$ 
stands for the usual integral average of $f$ over $Q$. 

\subsection*{Cantor family}
Fix $0<d<n$. 
Let $0<\dl<1$ be the solution to the equation 
$2^{n-d}(1-\dl)^d=1$, and 
consider the following Cantor family 
(cf. \cite{CPSS}):

\begin{itemize}
\item[$\bullet$]
Let $Q_0=E^0=[0, 1)^n$, and 
delete all but 
the $2^n$ corner cubes $\{Q_j^1\}$, 
of side $(1-\dl)/2$ to obtain $E^1$;
\item[$\bullet$]
Continue in this way, at the $k$ stage 
replacing each cube of $E^{k-1}$ by the $2^n$ corner cubes $\{Q_j^k\}$, 
of side $((1-\dl)/2)^k$ to get $E^k$.
\end{itemize}

Thus, $E^k$ contains $2^{nk}$ cubes 
of side $((1-\dl)/2)^k$, and hence 
\begin{equation}\label{1.2}
H^d(E^k)
\approx
2^{nk} \cdot ((1-\dl)/2)^{dk}
=
(2^{n-d}(1-\dl)^d)^k
=1.
\end{equation}

Let $\cS_0=\{Q_j^k\}$ be the Cantor family. 
Then, 
$$
\cA_{\cS_0}[\1_{Q_0}](x)
=
\sum_{j, k}\1_{Q_j^k}(x),
\quad x\in\R^n,
$$
is a~sparse operator of the function 
$\1_{Q_0}$ with $\eta=1-(1-\dl)^n$.
Since, 
$$
H^d(E^k) \approx 1,
\quad k=0,1,\ldots,
$$
we can not expect any 
$\cL^p(H^d)$, $p>0$, 
boundedness of the $\cA_{\cS_0}$.

In this paper 
we fix the target space of 
the sparse operator $\cA_{\cS}f$ 
defined on the Choquet space $\cL^p(H^d)$, $p\ge 1$.

\section{The first results}\label{sec2}
Fix $\cS\subset\cD$ be a~dyadic sparse family  
and we want to evaluate 
$$
\cA_{\cS}f(x)
=
\sum_{Q\in\cS}\fint_{Q}f\cdot\1_{Q}(x),
\quad f\ge 0,\,x\in\R^n.
$$

\subsection*{Duality argument}
We use a duality argument.
For suitable nonnegative functions 
$f$ and $g$,
\begin{align*}
\int_{\R^n}
g(x)\cA_{\cS}f(x)
\,{\rm d}x
&=
\sum_{Q\in\cS}
|Q|
\fint_{Q}f(x)\,{\rm d}x
\fint_{Q}g(x)\,{\rm d}x
\\ &\lesssim
\sum_{Q\in\cS}
|E_{Q}|
\fint_{Q}f(x)\,{\rm d}x
\fint_{Q}g(x)\,{\rm d}x
\\ &\le
\sum_{Q\in\cS}
\int_{E_{Q}}
Mf(x) \cdot Mg(x)
\,{\rm d}x
\\ &\le
\int_{\R^n}
Mf(x) \cdot Mg(x)
\,{\rm d}x.
\end{align*}
Here, 
$M$ is the usual dyadic Hardy-Littlewood maximal operator.
Using Adams' dual inequality \eqref{1.1} 
and H\"{o}lder's inequality,
\begin{align*}
\int_{\R^n}
g(x)\cA_{\cS}f(x)
\,{\rm d}x
&\lesssim
\int_{\R^n}
Mf(x) \cdot Mg(x)
\,{\rm d}x
\\ &\le
\int_{\R^n}
Mf(x)\cdot M_dMg(x)\,{\rm d}H^d
\\ &\le
\|Mf\|_{\cL^p(H^d)}
\|M_dMg\|_{\cL^{p'}(H^d)}
\\ &\lesssim
\|f\|_{\cL^p(H^d)}
\|M_dMg\|_{\cL^{p'}(H^d)},
\end{align*}
where we have used the boundedness of the Hardy-Littlewood maximal operator $M$ on $\cL^p(H^d)$, $p>d/n$, (see \cite{OV}).

\subsection*{Orlicz-Morrey spaces}
A~function 
$\Phi:\,[0, \8)\to[0, \8]$ 
is said to be a~Young function if 
it is left-continuous, convex and 
increasing, and if 
$\Phi(0)=0$ and 
$\Phi(t)\to\8$ as $t\to\8$. 

A~Young function $\Phi$ is said to satisfy the $\Dl_2$-condition, 
denoted $\Phi\in\Dl_2$, 
if for some $K>1$
$$
\Phi(2t)\le K\Phi(t)
\quad\text{for all }
t>0.
$$
Meanwhile, 
a~Young function $\Phi$ is said to satisfy the $\nabla_2$-condition, 
denoted $\Phi\in\nabla_2$, 
if for some $K>1$
$$
\Phi(t)\le\frac{1}{2K}\Phi(Kt)
\quad\text{ for all }
t>0.
$$
The function $\Phi(t)\equiv t$ 
satisfies the $\Dl_2$-condition 
but fails the $\nabla_2$-condition. 
If $1<p<\8$, then 
$\Phi(t)\equiv t^p$
satisfies both conditions. 

The complementary function $\ol{\Phi}$ of a Young function $\Phi$ 
is defined by
$$
\ol{\Phi}(t)
=
\sup\{ts-\Phi(s):\,s\in[0, \8)\}.
$$
Then $\ol{\Phi}$ is also a~Young function and 
$\ol{\ol{\Phi}}=\Phi$. 

Let $\Phi$ be a~Young function. 
For a~cube $Q$ define 
the $\Phi$-average over $Q$ of the measurable function $f$ by 
$$
\|f\|_{\Phi;Q}
=
\inf\lt\{\lm>0:\,
\fint_{Q}\Phi\lt(\frac{|f(x)|}{\lm}\rt)\,{\rm d}x\le 1
\rt\}.
$$
We define the dyadic fractional Orlicz maximal operator 
$M_{\al,\Phi}$ by
$$
M_{\al,\Phi}f(x)
=
\sup_{Q\in\cD}
\ell(Q)^{\al}
\|f\|_{\Phi;Q}\1_{Q}(x),
\quad 0<\al<n,\,x\in\R^n.
$$
It is well-known that (cf. \cite{SST}) 
$$
M_dMf(x) \approx M_{n-d,\Phi_0}f(x),
\qquad
\Phi_0(t)=t\log(e+t).
$$

\begin{quotation}
For $1<p\le\8$ and 
a~Young function $\Phi$, 
we define the Orlicz-Morrey space
$\cM^p_{\Phi}(H^d)$ 
by the set of all measurable functions $f$ satisfying 
$$
\|f\|_{\cM^p_{\Phi}(H^d)}
=
\|M_{n-d,\Phi}f\|_{\cL^p(H^d)}<\8.
$$
\end{quotation}

Since we notice that
\[
\|f\|_{\cM^{\8}_{\Phi}(H^d)}
=
\sup_{Q\in\cD}
\ell(Q)^{n-d}
\|f\|_{\Phi;Q},
\]
the space
$\cM^{\8}_{\Phi}(H^d)$ 
is a~usual Orlicz-Morrey space, 
which is introduced and studied in \cite{SST}.

\vspace{2mm}

We have had 
\begin{align*}
\int_{\R^n}
g(x)\cA_{\cS}f(x)
\,{\rm d}x
&\lesssim
\|f\|_{\cL^p(H^d)}
\|M_dMg\|_{\cL^{p'}(H^d)}
\\ &\approx
\|f\|_{\cL^p(H^d)}
\|M_{n-d,\Phi_0}g\|_{\cL^{p'}(H^d)}
\\ &=
\|f\|_{\cL^p(H^d)}
\|g\|_{\cM^{p'}_{\Phi_0}(H^d)}.
\end{align*}
In the left hand side, 
taking the supremum over 
$\|g\|_{\cM^{p'}_{\Phi_0}(H^d)}=1$,
we obtain the following theorem.

\begin{thm}\label{thm2.1}
For $1\le p<\8$, 
we have that
$$
\|\cA_{\cS}f\|_{{\cM^{p'}_{\Phi_0}(H^d)}'}
\lesssim
\|f\|_{\cL^p(H^d)}.
$$
\end{thm}

The function space 
${\cM^{p'}_{\Phi_0}(H^d)}'$ 
is a~associate space to 
the Orlicz-Morrey space 
$\cM^{p'}_{\Phi_0}(H^d)$ 
(for details, see the book \cite{BS}).

\begin{rem}\label{rem2.2}
By the definition, 
the Orlicz-Morrey space 
$\cM^{p'}_{\Phi_0}(H^d)$ 
is embedded into the Morrey space 
$\cM^{p'}(H^d)$, which is 
an associate dual space of 
the Choquet space $\cL^p(H^d)$. 
We remark that this embedding is proper.
Indeed, 
recall that the function 
$F(x):=\cA_{\cS_0}[\1_{Q_0}](x)$
(see Section \ref{sec1}).
Then, 
by the observation in Section \ref{sec1} 
and 
Theorem \ref{thm2.1}, 
one knows that 
$F\notin\cL^p(H^d)$
and
$F\in{\cM^{p'}_{\Phi_0}(H^d)}'$,
respectively. This means that 
the embedding 
$\cL^p(H^d)$ 
into 
${\cM^{p'}_{\Phi_0}(H^d)}'$ 
is proper, and hence
the embedding 
$\cM^{p'}_{\Phi_0}(H^d)$ 
into 
$\cM^{p'}(H^d)$ 
is also proper.
\end{rem}

\section{The second results}\label{sec3}
We say that $\cT\subset\cD$ is a~tiling of $\R^n$ 
if it satisfies 
$$
\1_{\R^n}(x)
=
\sum_{Q\in\cT}\1_{Q}(x),
\quad x\in\R^n.
$$
We fix the parameters 
$p\in[1, \8)$, $p'$, 
the Young function $\Phi\in\Dl_2$ and 
its complementary $\ol{\Phi}$.
Thanks to $\Phi\in\Dl_2$,
for any $Q\in\cD$ such that 
$\|f\|_{\Phi;Q}>0$,
we have that, see \cite[pages~77--78]{RaRe},
\begin{equation}\label{3.1}
\fint_{Q}\Phi\lt(\frac{|f(x)|}{\|f\|_{\Phi;Q}}\rt)\,{\rm d}x=1.
\end{equation}
We also fix the tile $\cT$. 
Define 
the Orlicz block space 
$\cB^p_{\Phi}(H^d,\cT)$ 
and 
the Orlicz-Morrey space 
$\cM^{p'}_{\ol{\Phi}}(H^d,\cT)$ 
by the set of all measurable functions
satisfying
\begin{align*}
\|f\|_{\cB^p_{\Phi}(H^d,\cT)}
&=
\lt\|
\sum_{Q\in\cT}
\|f\|_{\Phi;Q}^p
\1_{Q}
\rt\|_{\cL^1(H^d)}^{\frac1p}<\8,
\\
\|g\|_{\cM^{p'}_{\ol{\Phi}}(H^d,\cT)}
&=
\lt\|
\sum_{Q\in\cT}
\lt(
\ell(Q)^{n-d}\|g\|_{\ol{\Phi};Q}
\rt)^{p'}
\1_{Q}
\rt\|_{\cL^1(H^d)}^{\frac1{p'}}<\8,
\end{align*}
respectively.

\begin{thm}\label{thm3.1}
We have that
$$
\|f\|_{{\cM^{p'}_{\ol{\Phi}}(H^d)}'}
\lesssim
\|f\|_{\cB^p_{\Phi}(H^d,\cT)}
\lesssim
\|f\|_{{\cM^{p'}_{\ol{\Phi}}(H^d,\cT)}'}.
$$
\end{thm}

\begin{proof}[Proof of the first inequality]
For suitable nonnegative functions 
$f$ and $g$,
\begin{align*}
\int_{\R^n}f(x)g(x)\,{\rm d}x
&=
\sum_{Q\in\cT}
\int_{Q}f(x)g(x)\,{\rm d}x
\\ &=
\sum_{Q\in\cT}
|Q|
\fint_{Q}f(x)g(x)\,{\rm d}x
\\ &\le 2
\sum_{Q\in\cT}
|Q|
\|f\|_{\Phi;Q}
\|g\|_{\ol{\Phi};Q}
\\ &= 2
\int_{\R^n}
\lt(
\sum_{Q\in\cT}
\|f\|_{\Phi;Q}
\|g\|_{\ol{\Phi};Q}
\1_{Q}(x)
\rt)\,{\rm d}x
\\ &\le 4
\int_{\R^n}
\lt(
\sum_{Q\in\cT}
\|f\|_{\Phi;Q}\1_{Q}
\rt)
M_{n-d,\ol{\Phi}}g
\,{\rm d}H^d,
\\ \qquad
\text{we verify this inequality later},
\\ &\le 4
\lt\|
\sum_{Q\in\cT}
\|f\|_{\Phi;Q}\1_{Q}
\rt\|_{\cL^p(H^d)}
\cdot
\|M_{n-d,\ol{\Phi}}g\|_{\cL^{p'}(H^d)}
\\ &=4
\lt\|
\sum_{Q\in\cT}
\|f\|_{\Phi;Q}^p\1_{Q}
\rt\|_{\cL^1(H^d)}^{\frac1p}
\cdot
\|M_{n-d,\ol{\Phi}}g\|_{\cL^{p'}(H^d)}
\\ &=4
\|f\|_{\cB^p_{\Phi}(H^d,\cT)}
\cdot
\|M_{n-d,\ol{\Phi}}g\|_{\cL^{p'}(H^d)}.
\end{align*}
In the left hand side, 
taking the supremum over 
$\|g\|_{\cM^{p'}_{\ol{\Phi}}(H^d)}=1$,
we obtain the first inequality of 
Theorem \ref{thm3.1}.
\end{proof}

\noindent\textbf{Verification.}\quad
This inequality is a~consequence of Adams' dual inequality \eqref{1.1}.
We need the following observation.
For $Q_0\in\cD$,
$$
\int_{Q_0}
\sum_{\substack{
Q\subset Q_0: \\ Q\in\cT
}}
\|g\|_{\ol{\Phi};Q}\1_{Q}(y)
\,{\rm d}y
=
\sum_{\substack{
Q\subset Q_0: \\ Q\in\cT
}}
|Q|
\|g\|_{\ol{\Phi};Q}
=:{\rm(i)}.
$$
To proceed, we use another characterization of the Luxemburg norm 
(cf. \cite[p69]{RaRe}).
$$
\|g\|_{\ol{\Phi};Q}
\le
\inf_{s>0}
\lt\{
s+\frac{s}{|Q|}\int_{Q}
\ol{\Phi}\lt(\frac{g(x)}{s}\rt)\,{\rm d}x
\rt\}
\le 2
\|g\|_{\ol{\Phi};Q}.
$$
This entails
\begin{align*}
2|Q_0|
\|g\|_{\ol{\Phi};Q_0}
&\ge
\inf_{s>0}
\lt\{
s|Q_0|+s\int_{Q_0}
\ol{\Phi}\lt(\frac{g(x)}{s}\rt)\,{\rm d}x
\rt\}
\\ &=
\inf_{s>0}
\sum_{\substack{
Q\subset Q_0: \\ Q\in\cT
}}
\lt\{
s|Q|+s\int_{Q}
\ol{\Phi}\lt(\frac{g(x)}{s}\rt)\,{\rm d}x
\rt\}
\\ &\ge
\sum_{\substack{
Q\subset Q_0: \\ Q\in\cT
}}
\inf_{s>0}
\lt\{
s|Q|+s\int_{Q}
\ol{\Phi}\lt(\frac{g(x)}{s}\rt)\,{\rm d}x
\rt\}
\\ &\ge
\sum_{\substack{
Q\subset Q_0: \\ Q\in\cT
}}
|Q|
\|g\|_{\ol{\Phi};Q}
\\ &={\rm(i)}.
\end{align*}
Thus, we obtain
$$
\ell(Q_0)^{-d}
\int_{Q_0}
\sum_{\substack{
Q\subset Q_0: \\ Q\in\cT
}}
\|g\|_{\ol{\Phi};Q}\1_{Q}(y)
\,{\rm d}y
\le 2
\ell(Q_0)^{n-d}
\|g\|_{\ol{\Phi};Q_0}.
\quad\qed
$$

\begin{proof}[Proof of the second inequality]
We assume that
$$
0<\|f\|_{\cB^p_{\Phi}(H^d,\cT)}<\8.
$$
By Adams' dual theorem, 
there exists a~positive measure $\mu$ such that
$$
\|f\|_{\cB^p_{\Phi}(H^d,\cT)}^p
=
\sum_{Q\in\cT}
\|f\|_{\Phi;Q}^p\mu(Q),
$$
where the measure $\mu$ fulfills 
$$
\mu(Q)\le\ell(Q)^d
\quad\text{for all}\quad Q\in\cD.
$$
We invoke a fundamental dual equation:
$$
t\Phi'(t)
=
\Phi(t)+\ol{\Phi}(\Phi'(t)).
$$
This can be verified by the fact that
the equation 
$st=\Phi(t)+\ol{\Phi}(s)$
holds for $s=\Phi'(t)$
, see \cite[Theorem 8.12]{BS}.
It follows that, 
for any $Q\in\cD$ such that 
$\|f\|_{\Phi;Q}>0$,
\begin{align*}
\fint_{Q}
\frac{|f(x)|}{\|f\|_{\Phi;Q}}
\Phi'\lt(
\frac{|f(x)|}{\|f\|_{\Phi;Q}}
\rt)\,{\rm d}x
&=
\fint_{Q}\Phi\lt(
\frac{|f(x)|}{\|f\|_{\Phi;Q}}
\rt)\,{\rm d}x
+
\fint_{Q}\ol{\Phi}\lt(\Phi'\lt(
\frac{|f(x)|}{\|f\|_{\Phi;Q}}
\rt)\rt)\,{\rm d}x
\\ &=
1
+
\fint_{Q}\ol{\Phi}\lt(\Phi'\lt(
\frac{|f(x)|}{\|f\|_{\Phi;Q}}
\rt)\rt)\,{\rm d}x,
\end{align*}
where we have used \eqref{3.1}.
First we treat the quantity 
$\|f\|_{\Phi;Q}^p\mu(Q)$
times the left hand side:
$$
\|f\|_{\Phi;Q}^p\mu(Q)
\cdot
\fint_{Q}
\frac{|f(x)|}{\|f\|_{\Phi;Q}}
\Phi'\lt(
\frac{|f(x)|}{\|f\|_{\Phi;Q}}
\rt)\,{\rm d}x
=
\|f\|_{\Phi;Q}^{p-1}
\frac{\mu(Q)}{|Q|}
\int_{Q}|f(x)|f_{Q}(x)\,{\rm d}x,
$$
where, 
$$
f_{Q}(x)
:=
\Phi'\lt(
\frac{|f(x)|}{\|f\|_{\Phi;Q}}
\rt)\1_{Q}(x).
$$
Thus, we obtain
$$
\|f\|_{\Phi;Q}^p\mu(Q)
=
\int_{Q}|f(x)|F_{Q}(x)\,{\rm d}x,
$$
where, 
$$
F_{Q}(x)
:=
\lt[
1
+
\fint_{Q}\ol{\Phi}\lt(\Phi'\lt(
\frac{|f(x)|}{\|f\|_{\Phi;Q}}
\rt)\rt)\,{\rm d}x
\rt]^{-1}
\|f\|_{\Phi;Q}^{p-1}
\frac{\mu(Q)}{|Q|}
f_{Q}(x).
$$
It follows that 
$$
\sum_{Q\in\cT}
\|f\|_{\Phi;Q}^p\mu(Q)
=
\int_{\R^n}|f(x)|F(x)\,{\rm d}x,
$$
where, 
$$
F(x)
:=
\sum_{Q\in\cT}F_{Q}(x),
\quad x\in\R^n.
$$
We see that
\begin{align*}
\ell(Q)^{n-d}\|F_{Q}\|_{\ol{\Phi};Q}
&=
\ell(Q)^{n-d}
\frac{\mu(Q)}{|Q|}
\lt[
1
+
\fint_{Q}\ol{\Phi}\lt(\Phi'\lt(
\frac{|f(x)|}{\|f\|_{\Phi;Q}}
\rt)\rt)\,{\rm d}x
\rt]^{-1}
\|f\|_{\Phi;Q}^{p-1}
\|f_{Q}\|_{\ol{\Phi};Q}
\\ &\le 
\|f\|_{\Phi;Q}^{p-1},
\end{align*}
where we have used
$$
\ell(Q)^{n-d}
\frac{\mu(Q)}{|Q|}
\le 1
$$
and, see Lemma \ref{lem3.4}, below,
$$
\|f_{Q}\|_{\ol{\Phi};Q}
\le
1
+
\fint_{Q}\ol{\Phi}\lt(\Phi'\lt(
\frac{|f(x)|}{\|f\|_{\Phi;Q}}
\rt)\rt)\,{\rm d}x.
$$
Thus, we obtain
\begin{align*}
\|F\|_{\cM^{p'}_{\ol{\Phi}}(H^d,\cT)}
&=
\lt\|
\sum_{Q\in\cT}
\lt(
\ell(Q)^{n-d}\|F_{Q}\|_{\ol{\Phi};Q}
\rt)^{p'}
\1_{Q}
\rt\|_{\cL^1(H^d)}^{\frac1{p'}}
\\ &\le
\lt\|
\sum_{Q\in\cT}
\|f\|_{\Phi;Q}^p\1_{Q}
\rt\|_{\cL^1(H^d)}^{\frac1{p'}}
\\ &=
\|f\|_{\cB^p_{\Phi}(H^d,\cT)}^{\frac{p}{p'}},
\end{align*}
which implies
$$
\lt\|
F\|f\|_{\cB^p_{\Phi}(H^d,\cT)}^{-\frac{p}{p'}}
\rt\|_{\cM^{p'}_{\ol{\Phi}}(H^d,\cT)}
\le 1.
$$
This entails
\begin{align*}
\|f\|_{\cB^p_{\Phi}(H^d,\cT)}
&=
\|f\|_{\cB^p_{\Phi}(H^d,\cT)}^{-\frac{p}{p'}}
\cdot
\|f\|_{\cB^p_{\Phi}(H^d,\cT)}^p
\\ &\le
\int_{\R^n}
|f|
\cdot
F\|f\|_{\cB^p_{\Phi}(H^d,\cT)}^{-\frac{p}{p'}}
\,{\rm d}x
\\ &\le
\|f\|_{{\cM^{p'}_{\ol{\Phi}}(H^d,\cT)}'}.
\end{align*}
This is the second inequality of Theorem \ref{thm3.1}.
\end{proof}

Clearly, one has that
$$
\sup_{\cT}
\|f\|_{\cM^{\8}_{\ol{\Phi}}(H^d,\cT)}
\approx
\|f\|_{\cM^{\8}_{\ol{\Phi}}(H^d)},
$$
which yields the following corollary.

\begin{cor}\label{cor3.2}
We have that
$$
\inf_{\cT}
\|f\|_{\cB^1_{\Phi}(H^d,\cT)}
=
\inf_{\cT}
\lt\|
\sum_{Q\in\cT}
\|f\|_{\Phi;Q}\1_{Q}
\rt\|_{\cL^1(H^d)}
\approx
\|f\|_{{\cM^{\8}_{\ol{\Phi}}(H^d)}'}.
$$
\end{cor}

Theorem \ref{thm2.1} and Corollary \ref{cor3.2} give us the following theorem.

\begin{thm}\label{thm3.3}
We have that
\begin{equation}\label{3.2}
\inf_{\cT}
\|\cA_{\cS}f\|_{\cB^1_{\ol{\Phi}_0}(H^d,\cT)}
=
\inf_{\cT}
\lt\|
\sum_{Q\in\cT}
\|\cA_{\cS}f\|_{\ol{\Phi}_0;Q}
\1_{Q}
\rt\|_{\cL^1(H^d)}
\lesssim
\|f\|_{\cL^1(H^d)},
\qquad
\ol{\Phi}_0(t) \approx e^t-1,
\quad t>1.
\end{equation}
\end{thm}

\begin{lem}\label{lem3.4}
Let $\Phi$ be a~Young function.
\begin{itemize}
\item[{\rm(i)}]
$\ds
\|f\|_{\cL^{\Phi}}\le 1
\Longrightarrow 
\int_{\R^n}
\Phi(|f(x)|)\,{\rm d}x
\le
\|f\|_{\cL^{\Phi}};
$
\item[{\rm(ii)}]
$\ds
\|f\|_{\cL^{\Phi}}>1
\Longrightarrow 
\int_{\R^n}
\Phi(|f(x)|)\,{\rm d}x
\ge
\|f\|_{\cL^{\Phi}}.
$
\end{itemize}

\noindent
Therefore,
$$
\|f\|_{\cL^{\Phi}}
\le
\max\lt\{
1,
\int_{\R^n}
\Phi(|f(x)|)\,{\rm d}x
\rt\}.
$$
\end{lem}

\begin{proof}
In general, 
for a Young function $\Phi$ 
it should be remarked that 
$$
\lt\{\begin{array}{ll}
\theta\Phi(t)\ge\Phi(\theta t)&\mbox{ if $0<\theta<1$},
\\
\theta\Phi(t)\le\Phi(\theta t)&\mbox{ if $1<\theta<\infty$}.
\end{array}\rt.
$$
Indeed, since $\Phi$ is convex and $\Phi(0)=0$, 
for $0<\theta<1$, 
$$
\Phi(\theta t)
=
\Phi((1-\theta)0+\theta t)
\le
(1-\theta)\Phi(0)+\theta\Phi(t)
=
\theta\Phi(t)
$$
and, for $1<\theta<\8$, 
$$
\theta\Phi(t)
=
\theta\Phi(\theta^{-1}\theta t)
\le
\theta\theta^{-1}\Phi(\theta t)
=
\Phi(\theta t).
$$

When 
$0<\|f\|_{\cL^{\Phi}}\le 1$,
$$
\frac{1}{\|f\|_{\cL^{\Phi}}}
\int_{\R^n}\Phi(|f(x)|)\,{\rm d}x
\le
\int_{\R^n}
\Phi\lt(\frac{|f(x)|}{\|f\|_{\cL^{\Phi}}}\rt)
\,{\rm d}x
\le 1.
$$
When 
$1<\|f\|_{\cL^{\Phi}}<\8$,
for any 
$1<t_0<\|f\|_{\cL^{\Phi}}$, 
$$
t_0
\le
t_0
\int_{\R^n}
\Phi\lt(\frac{|f(x)|}{t_0}\rt)
\,{\rm d}x
\le
\frac{t_0}{t_0}
\int_{\R^n}\Phi(|f(x)|)\,{\rm d}x
=
\int_{\R^n}\Phi(|f(x)|)\,{\rm d}x.
$$
Choosing $t_0$ 
close to $\|f\|_{\cL^{\Phi}}$, 
we have that
$$
\|f\|_{\cL^{\Phi}}
\le
\int_{\R^n}\Phi(|f(x)|)\,{\rm d}x,
$$
Which completes the proof.
\end{proof}

\subsection*{
Estimation of the function
$\cA_{\cS_0}[\1_{Q_0}](x)$}
Let $Q_{0}=[0,1)^{n}$ and let $\cS_0=\{Q_j^k\}$ be the Cantor family. 
(See Section \ref{sec1}.) Then, 
$$
F(x)
:=
\cA_{\cS_0}[\1_{Q_0}](x)
=
\sum_{j, k}\1_{Q_j^k}(x),
\quad x\in\R^n,
$$
is a~sparse operator of the function 
$\1_{Q_0}$ with $\eta=1-(1-\dl)^n$.
We estimate the left hand side of \eqref{3.2} for the function $F$.

In this case, 
one should take $\cD=\cS_0$, and hence, 
the tiling $\cT$ is a~single cube
$Q_0=E^0$. 
Recall that $0<\dl<1$ is the solution to the equation 
$2^{n-d}(1-\dl)^d=1$. 
For later use, we fix $\lm_0>0$ so that
$$
\Lm_0:=e^{\frac1{\lm_0}}(1-\dl)^n<1.
$$
For $\lm>\lm_0$,
we shall estimate
$$
{\rm(i)}:=\int_{Q_0}\ol{\Phi}_0\lt(\frac{F(x)}{\lm}\rt)\,{\rm d}x.
$$
Since we may take
$\ol{\Phi}_{0}(t)=e^{t}-1, t>0$,
it follows that
$$
{\rm(i)}=\sum_{l=1}^{\8}
\frac1{l!}
\int_{Q_0}\lt(\frac{F(x)}{\lm}\rt)^l\,{\rm d}x.
$$
By the layer cake representation,
we see that
$$
\int_{Q_0}\lt(\frac{F(x)}{\lm}\rt)^l\,{\rm d}x
\le\frac{l}{\lm}
\sum_{k=0}^{\8}
\lt(\frac{k+1}{\lm}\rt)^{l-1}
|E^k|.
$$
Noticing that, cf. \eqref{1.2}, 
$$
|E^k|
=
2^{nk} \cdot ((1-\dl)/2)^{nk}
=
(1-\dl)^{nk},
$$
we obtain
\begin{align*}
{\rm(i)}
&\le
\frac1{\lm}
\sum_{k=0}^{\8}
e^{\frac{k+1}{\lm}}|E^k|
\\ &=
\frac{1}{\lm}
\cdot
e^{\frac1{\lm}}
\sum_{k=0}^{\8}
\lt(e^{\frac1{\lm}}(1-\dl)^n\rt)^k
\\ &\le
\frac{1}{\lm}
\cdot
e^{\frac1{\lm_0}}
\sum_{k=0}^{\8}(\Lm_0)^k
\\ &=
\frac{1}{\lm}
\cdot
\frac{e^{\frac1{\lm_0}}}{1-\Lm_0},
\end{align*}
where we have used $\lm>\lm_{0}$
in the last inequality.
If we take 
$\lm=\ds\frac{e^{\frac1{\lm_0}}}{1-\Lm_0}$
so that the right hand side equals to $1$, 
which means that 
$\lm$ majorizes the Luxemburg norm of $F$.


\begin{thebibliography}{999}

\bibitem{Ad} D.~R.~Adams, 
\emph{Choquet integrals in potential theory},
Publ.~Mat. \textbf{42}~(1998), 3--66.

\bibitem{BS}
C.~Bennett and R.~Sharpley, 
\emph{{it Interpolation of Operators}},
Academic Press, 1988.

\bibitem{CPSS} 
M.~Carro, C.~P\'{e}rez, F.~Soria and J.~Soria,
\emph{Maximal functions and the control of weighted inequalities for the fractional integral operator}, 
Indiana Univ. Math. J., \textbf{54}~(2005), no.~3, 627--644. 

\bibitem{Le} A.~K.~Lerner,
\emph{On pointwise estimates involving sparse operators},
New York J. Math., \textbf{22}~(2016), 341--349.

\bibitem{RaRe} M.~M.~Rao and D.~Z.~Ren, 
\emph{{\it Theory of Orlicz Spaces}}, 
Dekker 1991.

\bibitem{ST}
H.~Saito and H.~Tanaka,
\emph{Dual of the Choquet spaces with general Hausdorff content},
Studia Math., \textbf{266}~(2022), 323--335.

\bibitem{STW} 
H.~Saito, H.~Tanaka and T.~Watanabe,
\emph{Abstract dyadic cubes and the dyadic maximal operator with the Hausdorff content},
Bull. Sci Math., \textbf{140}~(2016), 757--773.

\bibitem{SST} 
Y.~Sawano, S.~Sugano and H.~Tanaka, 
\emph{Orlicz-Morrey spaces and fractional operators}, 
Potential Anal. \textbf{36}~(2012), no.~4, 517--556.

\bibitem{OV}
J.~Orobitg and J.~Verdera, 
\emph{Choquet integrals, Hausdorff content and the Hardy-Littlewood maximal  operator}, 
Bull. London Math. Soc. \textbf{30}~(1998), no.~2, 145--150.

\end{thebibliography}
\end{document}